\documentclass[12pt]{article}
\usepackage{amsmath, amssymb, amsfonts}
\usepackage{amsthm}
\usepackage[utf8]{inputenc}
\usepackage{csquotes}
\usepackage{geometry}
\usepackage{adjustbox}
\usepackage{graphicx}
\usepackage{verbatim}
\usepackage{array}
\usepackage{enumitem}
\usepackage{booktabs}
\usepackage{hyperref}
\usepackage{url}
\usepackage[capitalize]{cleveref}
\usepackage{tikz}

\title{Wave Arithmetic: A Smooth Integral Representation of Number Theory}
\author{Stanislav Semenov \\
\href{mailto:stas.semenov@gmail.com}{stas.semenov@gmail.com} \\
\href{https://orcid.org/0000-0002-5891-8119}{ORCID: 0000-0002-5891-8119}}
\date{April 21, 2025}

\theoremstyle{definition}
\newtheorem{definition}{Definition}[section]

\newtheorem{example}{Example}[section]

\theoremstyle{plain}
\newtheorem{axiom}{Axiom}[section]
\newtheorem{theorem}[definition]{Theorem}

\newtheorem{proposition}[definition]{Proposition}

\newtheorem{hypothesis}[definition]{Hypothesis}

\theoremstyle{remark}
\newtheorem*{remark}{Remark}

\begin{document}

\maketitle

\begin{abstract}
We introduce \textit{Wave Arithmetic}, a smooth analytical framework in which natural, integer, and rational numbers are represented not as discrete entities, but as integrals of smooth, compactly supported or periodic kernel functions. In this formulation, each number arises as the accumulated amplitude of a structured waveform---an interference pattern encoded by carefully designed kernels. Arithmetic operations such as addition, multiplication, and exponentiation are realized as geometric and tensorial constructions over multidimensional integration domains. Rational numbers emerge through amplitude scaling, and negative values through sign inversion, preserving all classical arithmetic identities within a continuous and differentiable structure. This representation embeds number theory into the realm of smooth analysis, enabling new interpretations of primality, factorization, and divisibility as geometric and spectral phenomena. Beyond technical formulation, \textit{Wave Arithmetic} proposes a paradigm shift: numbers as the collapsed states of harmonic processes---analytic resonances rather than atomic symbols.
\end{abstract}

\subsection*{Mathematics Subject Classification}
03F60 (Constructive and recursive analysis), 26E40 (Constructive analysis), 03F03 (Proof theory and constructive mathematics)

\subsection*{ACM Classification}
F.4.1 Mathematical Logic, F.1.1 Models of Computation

\section*{Introduction}

Classical arithmetic is fundamentally discrete. Natural numbers are atomic symbols; operations such as addition and multiplication are defined axiomatically or recursively, and numerical identities are established through symbolic manipulation. While this framework underpins modern mathematics and computation, it lacks smoothness, continuity, and structural compatibility with analysis, geometry, and signal theory. In particular, the transition from arithmetic to analysis typically requires external embedding---through step functions, distributions, or real-number codings---rather than an intrinsic reformulation of arithmetic itself~\cite{rudin1976principles,folland1999real}.

This work introduces \textit{Wave Arithmetic}---a smooth, analytic framework in which numbers are expressed as integrals of carefully constructed kernel functions, such as compactly supported bumps or periodic oscillators. In this model, a number is not a symbol, but the accumulated amplitude of a waveform. Arithmetic operations become tensorial or geometric constructions over multidimensional integration domains. For example, addition corresponds to one-dimensional integral summation, multiplication to two-dimensional convolution-like grid integration, and exponentiation to multidimensional tensor product integration. Rational numbers are represented via amplitude scaling, and negative numbers via signed inversion of the kernel.

The motivation is twofold. First, to construct an analytic foundation of arithmetic that is differentiable, smooth, and compatible with tools from harmonic analysis~\cite{stein2003fourier,gowers2001fourier}, geometry, and signal processing. Second, to reinterpret the notion of number not as a primitive static object, but as the emergent result of harmonic accumulation~\cite{stein2003fourier,berry1999riemann} and geometric configuration---a waveform collapsed into scalar identity. This paradigm shift opens the door to new analytic diagnostics for properties such as primality, divisibility, and factorization, all recast as structural or spectral features of the associated kernel integrals.

Compared to existing approaches, \textit{Wave Arithmetic} differs fundamentally in both construction and philosophy:
\begin{itemize}
  \item In contrast to recursive or symbolic arithmetic (Peano, Church, lambda-calculus), our model operates on continuous, smooth functions with localized or periodic structure.
  \item Unlike computable analysis or real arithmetic, which embeds integers as reals or sequences, we preserve exact arithmetic identities at the level of integrals.
  \item Unlike mollification or regularization techniques, where smooth functions approximate discontinuous behavior, here the smooth function \emph{is} the number---not an approximation, but its analytical instantiation.
  \item Unlike Fourier or wavelet decompositions used to analyze existing functions, we construct numbers themselves as harmonic superpositions.
\end{itemize}

The model preserves all fundamental arithmetic laws through integral and tensor identities, while introducing analytic notions such as \emph{analytic primality}, \emph{residual flattening error}, and \emph{spectral rigidity}. These concepts enable a novel interpretation of number-theoretic properties through the lens of multidimensional geometry and interference patterns~\cite{connes1998trace,berry1999riemann}.

Ultimately, \textit{Wave Arithmetic} proposes that arithmetic is not merely a symbolic calculus, but a geometric and harmonic phenomenon: numbers emerge as resonant accumulations of structured energy. The implications extend beyond theoretical reformulation—they suggest new tools for analytical number theory, signal-based cryptography, smooth numerical modeling, and even philosophical reinterpretations of numerical identity.

\section{Basic Definitions and Analytic Constructions}
\label{sec:basic-definitions}

We construct an arithmetic model of \( \mathbb{N} \), \( \mathbb{Z} \), and \( \mathbb{Q} \) based on integrals of smooth, compactly supported functions. Each number is represented as the integral of one or more bump functions, with arithmetic operations defined geometrically through dimension-aligned integration.

\subsection{Base Kernel Function}
\label{subsec:base-kernel}

Let \( \phi(x) \in C^\infty(\mathbb{R}) \) be a smooth bump function supported on \( \left[-\tfrac{1}{2}, \tfrac{1}{2}\right] \), centered at 0 and normalized by:
\[
\int_{\mathbb{R}} \phi(x)\,dx = 1
\]
Define the shifted version:
\[
\phi_k(x) := \phi(x - (k - 1)), \quad k \in \mathbb{Z}
\]
Each \( \phi_k \) has support in \( [k - \tfrac{3}{2}, k - \tfrac{1}{2}] \).

\subsection{Analytic Natural Numbers}
\label{subsec:natural-numbers}

We define the analytic representation of a natural number \( n \in \mathbb{N} \) as:
\[
\mathbb{N}_a(n) := \sum_{k=1}^n \int_{\mathbb{R}} \phi_k(x)\,dx = n
\]

\subsection{Addition}
\label{subsec:addition}

Addition corresponds to 1D extension along the \( x \)-axis:
\[
\mathbb{N}_a(a + b) := \sum_{k=1}^{a+b} \int_{\mathbb{R}} \phi_k(x)\,dx = a + b
\]

\subsection{Multiplication}
\label{subsec:multiplication}

Define the 2D kernel as:
\[
\phi_{i,j}(x,y) := \phi(x - i + 1) \cdot \phi(y - j + 1)
\]
Then:
\[
\mathbb{N}_a(a \cdot b) := \sum_{i=1}^a \sum_{j=1}^b \iint_{\mathbb{R}^2} \phi_{i,j}(x,y)\,dx\,dy = ab
\]

\subsection{Exponentiation}
\label{subsec:exponentiation}

Let \( b \in \mathbb{N} \). Define the \( b \)-dimensional kernel:
\[
\phi_{\vec{i}}(\vec{x}) := \prod_{k=1}^b \phi(x_k - i_k + 1), \quad \vec{i} \in [1,a]^b
\]
Then:
\[
\mathbb{N}_a(a^b) := \sum_{\vec{i} \in [1,a]^b} \int_{\mathbb{R}^b} \phi_{\vec{i}}(\vec{x})\,d\vec{x} = a^b
\]

\subsection{Rational Numbers, Zero, Negatives, and Division}
\label{subsec:rationals-division}

We extend the system to \( \mathbb{Q} \) by allowing amplitude scaling and sign inversion.

\subsubsection{Zero and Negative Numbers}
\label{subsubsec:zero-negatives}

Define the analytic zero as the zero function:
\[
\mathbb{Q}_a(0) := \int_{\mathbb{R}} 0 \cdot \phi(x)\,dx = 0
\]
A negative number \( -r \in \mathbb{Q} \) is represented by:
\[
\mathbb{Q}_a(-r) := -\mathbb{Q}_a(r) = \int_{\mathbb{R}} (-r) \cdot \phi(x)\,dx = -r
\]

\subsubsection{Rational Numbers via Scaled Bumps}
\label{subsubsec:scaled-rationals}

Let \( \phi^{(r)}(x) := r \cdot \phi(x) \) with \( r \in \mathbb{Q} \). Then:
\[
\int \phi^{(r)}(x)\,dx = r
\]
The analytic rational number \( \frac{m}{n} \in \mathbb{Q} \) is:
\[
\mathbb{Q}_a\left(\tfrac{m}{n}\right) := \sum_{k=1}^{|m|} \int_{\mathbb{R}} \operatorname{sgn}(m) \cdot \phi_k^{(1/n)}(x)\,dx = \tfrac{m}{n}
\]

\subsubsection{Division and the Multiplicative Inverse}
\label{subsubsec:division}

We define division by using amplitude scaling:
\[
\frac{1}{b} := \int_{\mathbb{R}} \phi_1^{(1/b)}(x)\,dx \quad \text{where } \phi_1^{(1/b)}(x) := \tfrac{1}{b} \cdot \phi(x)
\]
Then:
\[
\frac{a}{b} := \sum_{k=1}^a \int_{\mathbb{R}} \phi_k^{(1/b)}(x)\,dx = \frac{a}{b}
\]
More generally:
\[
\frac{m}{n} = \sum_{k=1}^{|m|} \int_{\mathbb{R}} \operatorname{sgn}(m) \cdot \phi_k^{(1/n)}(x)\,dx
\]

\subsection{Summary}
\label{subsec:summary}

\begin{itemize}
  \item \textbf{Natural numbers} are represented as sums of unit bump integrals in \( \mathbb{R}^1 \).
  \item \textbf{Addition, multiplication, exponentiation} correspond to integration over increasing dimensions.
  \item \textbf{Rational numbers} are realized by scaling the amplitude of bumps.
  \item \textbf{Zero} is the zero function; \textbf{negatives} are bumps with negative amplitude.
  \item \textbf{Division} is defined by scaled integration, producing the multiplicative inverse.
\end{itemize}

\section{Functional Axioms of Smooth Arithmetic}
\label{sec:functional-axioms}

\begin{axiom}[Smooth Base Function Axiom]
There exists a fixed smooth function \( \phi \in C_c^\infty(\mathbb{R}) \) such that:
\begin{itemize}
  \item \( \mathrm{supp}(\phi) \subseteq [-\tfrac{1}{2}, \tfrac{1}{2}] \),
  \item \( \phi(x) \ge 0 \) for all \( x \in \mathbb{R} \),
  \item \( \int_{\mathbb{R}} \phi(x) \, dx = 1 \).
\end{itemize}
\end{axiom}

\begin{axiom}[Analytic Representation of Natural Numbers]
\label{ax:analytic-naturals}
Each natural number \( n \in \mathbb{N} \) is represented by:
\[
\mathbb{N}_a(n) := \sum_{k=1}^n \int_{\mathbb{R}} \phi_k(x) \, dx,
\]
where \( \phi_k(x) := \phi(x - (k - 1)) \).
\end{axiom}

\begin{axiom}[Additivity of Integrals for Disjoint Bumps]
\label{ax:additivity}
For all \( a, b \in \mathbb{N} \), the representation satisfies:
\[
\mathbb{N}_a(a + b) = \mathbb{N}_a(a) + \mathbb{N}_a(b),
\]
with disjoint supports for \( \phi_k \).
\end{axiom}

\begin{axiom}[Tensor Product Multiplication]
\label{ax:product-multiplication}
For \( a, b \in \mathbb{N} \), the product is represented as:
\[
\mathbb{N}_a(a \cdot b) := \sum_{i=1}^a \sum_{j=1}^b \iint_{\mathbb{R}^2} \phi_{i,j}(x,y)\,dx\,dy,
\]
where \( \phi_{i,j}(x,y) := \phi(x - i + 1) \cdot \phi(y - j + 1) \).
\end{axiom}

\begin{axiom}[Exponentiation via Multidimensional Kernel]
\label{ax:exponentiation}
For \( a, b \in \mathbb{N} \), define:
\[
\mathbb{N}_a(a^b) := \sum_{\vec{i} \in [1,a]^b} \int_{\mathbb{R}^b} \prod_{k=1}^b \phi(x_k - i_k + 1) \, d\vec{x}.
\]
\end{axiom}

\begin{axiom}[Signed and Scaled Representation of Rationals]
\label{ax:rationals}
For \( r = \frac{m}{n} \in \mathbb{Q} \), define:
\[
\mathbb{Q}_a(r) := \sum_{k=1}^{|m|} \int_{\mathbb{R}} \operatorname{sgn}(m) \cdot \tfrac{1}{n} \cdot \phi_k(x) \, dx.
\]
\end{axiom}

\begin{axiom}[Linearity and Inversion]
\label{ax:linearity-inversion}
The analytic representation satisfies the following two properties:
\begin{itemize}
  \item Linearity: \( \mathbb{Q}_a(r + s) = \mathbb{Q}_a(r) + \mathbb{Q}_a(s) \),
  \item Inversion: \( \mathbb{Q}_a(-r) = -\mathbb{Q}_a(r) \).
\end{itemize}
\end{axiom}

\begin{axiom}[Zero as Null Integral]
\label{ax:zero}
\[
\mathbb{Q}_a(0) := \int_{\mathbb{R}} 0 \cdot \phi(x) \, dx = 0.
\]
\end{axiom}

\section{Algebraic Axioms of Smooth Arithmetic}
\label{sec:algebraic-axioms}

\begin{axiom}[Commutativity of Addition]
\label{ax:add-comm}
\[
\mathbb{N}_a(a + b) = \mathbb{N}_a(b + a)
\]
\end{axiom}

\begin{axiom}[Associativity of Addition]
\label{ax:add-assoc}
\[
\mathbb{N}_a((a + b) + c) = \mathbb{N}_a(a + (b + c))
\]
\end{axiom}

\begin{axiom}[Commutativity of Multiplication]
\label{ax:mul-comm}
\[
\mathbb{N}_a(a \cdot b) = \mathbb{N}_a(b \cdot a)
\]
\end{axiom}

\begin{axiom}[Associativity of Multiplication]
\label{ax:mul-assoc}
\[
\mathbb{N}_a((a \cdot b) \cdot c) = \mathbb{N}_a(a \cdot (b \cdot c))
\]
\end{axiom}

\begin{axiom}[Distributivity]
\label{ax:distributive}
\[
\mathbb{N}_a(a \cdot (b + c)) = \mathbb{N}_a(a \cdot b) + \mathbb{N}_a(a \cdot c)
\]
\end{axiom}

\begin{axiom}[Neutral Elements]
\label{ax:neutral}
There exist:
\begin{itemize}
  \item \( 0 \in \mathbb{Q}_a \) such that \( \mathbb{Q}_a(a + 0) = \mathbb{Q}_a(a) \),
  \item \( 1 \in \mathbb{N}_a \) such that \( \mathbb{N}_a(a \cdot 1) = \mathbb{N}_a(a) \).
\end{itemize}
\end{axiom}

\begin{remark}
It is important to note that the algebraic axioms listed below are not intended as new axioms for the classical number systems \( \mathbb{N} \), \( \mathbb{Z} \), or \( \mathbb{Q} \). Rather, they serve as postulates that ensure the analytic model constructed via integrals of smooth functions faithfully preserves the fundamental algebraic properties of these systems. These axioms guarantee that addition, multiplication, and other operations in the analytic model remain compatible with their discrete counterparts.
\end{remark}

\section{Analytic Primes}
\label{sec:analytic-primes}

In classical arithmetic, a prime number is a natural number greater than 1 with no nontrivial divisors. We extend this notion to the analytic model by leveraging the geometric structure of integral representations.

\begin{definition}[Analytic Prime Number]
A number \( p \in \mathbb{N} \) is an \emph{analytic prime} if its representation \( \mathbb{N}_a(p) \) admits no nontrivial decomposition into a tensor product of lower-dimensional integrals. Formally, there exist no \( a, b \in \mathbb{N} \) with \( a, b \ge 2 \) such that:
\[
\mathbb{N}_a(p) = \sum_{i=1}^a \sum_{j=1}^b \iint_{\mathbb{R}^2} \phi_{i,j}(x,y) \, dx \, dy.
\]
Equivalently, \( p \) cannot be expressed analytically as \( \mathbb{N}_a(a \cdot b) \) for \( a, b \ge 2 \).
\end{definition}

\begin{remark}[Geometric and Analytic Primality]
Analytic primality has both a geometric and analytic interpretation:
\begin{itemize}
  \item \textbf{Geometric:} For composite \( n = a \cdot b \), the analytic representation \( \mathbb{N}_a(n) \) decomposes into a full \( a \times b \) grid of disjoint bump functions \( \phi_{i,j}(x,y) \). For primes \( p \), the only admissible grids are degenerate: \( 1 \times p \) or \( p \times 1 \), corresponding to trivial multiplication by 1.
  \item \textbf{Analytic:} Due to the normalization \( \int \phi = 1 \), any analytic decomposition must preserve the total integral value. Nontrivial decompositions would require overlapping supports or non-integer scaling, violating both the disjointness and amplitude structure defined in the axioms.
\end{itemize}
\end{remark}

\begin{example}
The number \( 6 \) factors analytically as:
\[
\mathbb{N}_a(6) = \sum_{i=1}^2 \sum_{j=1}^3 \iint \phi_{i,j}(x,y) \, dx \, dy = \left(\sum_{i=1}^2 \int \phi_i(x) \, dx\right) \cdot \left(\sum_{j=1}^3 \int \phi_j(y) \, dy\right),
\]
which reflects the classical identity \( 6 = 2 \cdot 3 \).

In contrast, for \( \mathbb{N}_a(7) \), no such nontrivial decomposition exists. Any attempt to write \( 7 = a \cdot b \) with \( a, b \ge 2 \) would require \( a = 7/2 \), \( a = 7/3 \), etc., contradicting the requirement that \( a, b \in \mathbb{N} \). Therefore, \( \mathbb{N}_a(7) \) remains indivisible in the analytic sense.
\end{example}

\begin{proposition}
The set of analytic primes \( \mathcal{P}_a \) coincides with the set of classical primes \( \mathcal{P} \)~\cite{apostol1976introduction,serre1973course}.
\end{proposition}

\begin{proof}
We prove both directions:
\begin{enumerate}
  \item Let \( p \in \mathcal{P} \) be classically prime. Then by definition, \( p \) cannot be written as \( a \cdot b \) with \( a, b \ge 2 \). Hence, by Axiom ~\ref{ax:product-multiplication}, there is no corresponding analytic product decomposition of \( \mathbb{N}_a(p) \), so \( p \in \mathcal{P}_a \).
  \item Conversely, suppose \( p \notin \mathcal{P} \). Then \( p = a \cdot b \) for some \( a, b \ge 2 \), and therefore \( \mathbb{N}_a(p) \) admits a nontrivial 2D analytic decomposition. Thus \( p \notin \mathcal{P}_a \).
\end{enumerate}
\end{proof}

\begin{remark}
This correspondence confirms that the analytic model faithfully encodes classical arithmetic. It also opens a path toward formulating smooth or probabilistic notions of primality.
\end{remark}

Future directions include:
\begin{itemize}
  \item Construction of smooth approximations to the prime-counting function \( \pi(x) \),
  \item Development of analytic analogs to the Riemann zeta function,
  \item Exploration of smooth primality indicators and sieve methods,
  \item Defining probabilistic or entropy-based measures of primality using local geometry of integral structures.
\end{itemize}

\section{Analytic Factorization}
\label{sec:analytic-factorization}

The definition of analytic primality naturally leads to a broader concept: \emph{analytic factorization}. In this framework, the analytic representation of a composite number can be decomposed into tensor products of lower-dimensional integrals, each corresponding to an analytically irreducible (i.e., prime) component.

\begin{definition}[Analytic Factorization]
Let \( n \in \mathbb{N} \) and \( \mathbb{N}_a(n) \) its analytic representation via multidimensional integration. An \emph{analytic factorization} of \( n \) is a sequence of decompositions:
\[
\mathbb{N}_a(n) = \int_{\mathbb{R}^{d_1 + \cdots + d_k}} \prod_{i=1}^k \Phi_i(\vec{x}_i) \, d\vec{x}_1 \cdots d\vec{x}_k,
\]
where each \( \Phi_i \) is a bump function of dimension \( d_i \), and the integral corresponds to a product of analytic representations of irreducible factors \( n = p_1 \cdot \cdots \cdot p_k \).

Each component \( \mathbb{N}_a(p_i) \) in the analytic factorization corresponds to a one-dimensional integral (or a finite sum of such integrals) over scaled and shifted bump functions. This reflects the definition of analytic primality, where a prime number is characterized precisely by the irreducibility of such one-dimensional integral structures. Thus, analytic factorization decomposes a multidimensional integral into tensor products of elementary 1D analytic primes.
\end{definition}

\begin{remark}
This notion mirrors classical prime factorization, but in a geometric setting: each factor corresponds to a lower-dimensional integral kernel that cannot be further split under the constraints of smoothness, disjoint support, and normalization.
\end{remark}

\begin{example}
The number \( n = 12 = 2 \cdot 2 \cdot 3 \) admits the analytic decomposition:
\[
\mathbb{N}_a(12) = \left(\sum_{i=1}^2 \int \phi^{(2)}_i(x) \, dx\right) 
\times \left(\sum_{j=1}^2 \int \phi^{(2)}_j(y) \, dy\right) 
\times \left(\sum_{k=1}^3 \int \phi^{(3)}_k(z) \, dz\right).
\]
This corresponds to the classical prime factorization \( 12 = 2 \times 2 \times 3 \), where each prime factor \( p_i \) is represented by a one-dimensional analytic component \( \mathbb{N}_a(p_i) \). The superscripts label independent coordinates.
\end{example}

\begin{remark}[Uniqueness and Order]
In classical arithmetic, factorizations are unique up to permutation of prime factors. In the analytic model, different groupings and orderings of the tensor integrals (e.g., \( (2 \times 2) \times 3 \) vs. \( 2 \times (2 \times 3) \)) may result in different geometric constructions, but the total integral remains invariant due to the linearity and associativity of integration.

For example, \( n = 8 \) admits both:
\begin{itemize}
  \item a cubic representation: \( \mathbb{N}_a(2) \times \mathbb{N}_a(2) \times \mathbb{N}_a(2) \),
  \item and a rectangular prism: \( \mathbb{N}_a(4) \times \mathbb{N}_a(2) \),
\end{itemize}
reflecting different valid decompositions that ultimately yield the same analytic value.
\end{remark}

\begin{remark}[On Wave Profiles and Superscripts]
In the analytic representation of numbers via integrals of wave kernels \( \phi \), we sometimes write \( \phi^{(p)}_i \) to indicate that the function corresponds to the \( i \)-th harmonic contribution associated with the prime factor \( p \). The superscript \( (p) \) distinguishes the spectral or geometric profile of that kernel — for example, its characteristic frequency, phase, or localization scale.

This is especially relevant in composite numbers, where the full analytic structure arises as a tensor product of contributions from different prime factors. In such cases, each \( \phi^{(p)}_i \) may live in its own coordinate axis or frequency domain, contributing independently to the overall resonance pattern.

In contrast, when defining an \emph{analytic prime}, we do not require such superscripts: the essential idea is the \emph{indivisibility} of the integral structure. The non-decomposability of \( \mathbb{N}_a(p) \) into a product of lower-dimensional integrals captures the same essence as irreducibility in classical number theory — without referring to the internal profile of \( \phi \).
\end{remark}

\subsection*{Analytic Divisibility and Multiplicativity}

\begin{definition}[Analytic Divisor]
Let \( a, n \in \mathbb{N} \). We say that \( a \) is an \emph{analytic divisor} of \( n \), denoted \( a \mid_a n \), if there exists \( b \in \mathbb{N} \) such that:
\[
\mathbb{N}_a(n) = \mathbb{N}_a(a) * \mathbb{N}_a(b),
\]
where \( * \) denotes a tensor product of integrals over disjoint coordinate domains, consistent with the structure defined in Axiom~\ref{ax:product-multiplication}.
\end{definition}

\begin{remark}
The analytic divisor relation \( a \mid_a n \) respects both the arithmetic and geometric structure of the model. It coincides with classical divisibility \( a \mid n \) under the assumption of integer-valued inputs and disjoint bump supports, and admits a structural interpretation: the integral kernel representing \( n \) decomposes cleanly into a substructure corresponding to the analytic representation of \( a \).
\end{remark}

\begin{definition}[Analytic Multiplicativity]
The analytic representation is said to be \emph{multiplicative} if for all \( a, b \in \mathbb{N} \), the following holds:
\[
\mathbb{N}_a(a \cdot b) = \mathbb{N}_a(a) * \mathbb{N}_a(b),
\]
where the right-hand side denotes the analytic construction over a product space \( \mathbb{R}^{d_a + d_b} \), assuming disjoint variables.
\end{definition}

\begin{remark}
This property follows directly from the tensor product formulation of multiplication (Axiom ~\ref{ax:product-multiplication}), and ensures that the analytic model behaves as a semiring under bump-integral operations.
\end{remark}

All fundamental properties of analytic factorization coincide with those of classical arithmetic, provided that the factors correspond to integer values and the supports of bump functions remain disjoint and properly normalized.

\section{Approximate Decomposability and Analytic Rigidity}
\label{sec:analytic-rigidity}

In classical arithmetic, primality is defined in terms of the absence of nontrivial integer factorizations. In the analytic framework, we may attempt to recover primality from first principles---without reference to classical number theory---by asking whether a number's integral representation admits any low-error approximation by a tensor product of lower-dimensional bump structures.

\begin{definition}[Approximate Decomposability Functional]
Let \( n \in \mathbb{N} \). Define the approximation error functional
\[
E(n; a, b) := \left\| \mathbb{N}_a(n) - \mathbb{N}_a(a) * \mathbb{N}_a(b) \right\|_{L^2}, \quad a,b \ge 2.
\]
We say that \( n \) is \emph{approximately decomposable} if there exist real \( a,b \ge 2 \) with \( a \cdot b = n \) such that \( E(n; a, b) < \epsilon \) for some fixed small threshold \( \epsilon > 0 \). Otherwise, \( n \) is said to be \emph{analytically rigid}.
\end{definition}

\begin{definition}[Analytic Rigidity Criterion]
A number \( n \in \mathbb{N} \) is \emph{analytically prime} if
\[
\inf_{\substack{a, b \ge 2 \\ a \cdot b = n}} E(n; a, b) > \epsilon,
\]
for some fixed \( \epsilon > 0 \) depending on the bump kernel \( \phi \).
\end{definition}

\begin{remark}[Geometric Intuition]
The function \( \mathbb{N}_a(n) \) consists of \( n \) disjoint unit bump functions aligned at integer positions. An analytic decomposition into a product \( a \cdot b \) requires these bumps to be arranged in a rectangular \( a \times b \) grid. For prime \( n \), no such grid exists, and any attempt to approximate it via non-integer \( a, b \) results in either overlap or fragmentation, both incompatible with the axioms of the model.
\end{remark}

\begin{example}[Contrasting 16 and 17]
For \( n = 16 \), the representation \( \mathbb{N}_a(4) * \mathbb{N}_a(4) \) yields a perfect match: \( E(16; 4, 4) = 0 \), corresponding to a \( 4 \times 4 \) rectangular grid.

For \( n = 17 \), no such exact decomposition exists. Any attempt to find \( a, b \in \mathbb{N} \) with \( a, b \ge 2 \) and \( a \cdot b = 17 \) fails, as no such pair exists. Attempts to approximate with non-integer values, such as \( a = b = \sqrt{17} \), lead to irregular bump arrangements requiring fractional translations, which violate disjointness and unit alignment.

Note that not all composite numbers admit square decompositions: for example, \( 10 = 2 \cdot 5 \) corresponds to a \( 2 \times 5 \) rectangular grid. Hence, the failure of square decomposition alone is not sufficient to indicate primality; the absence of \emph{any} integer-factorable grid is the correct analytic criterion.
\end{example}

\begin{remark}[Potential for Analytic Primality Testing]
This framework suggests a purely analytic method for detecting prime numbers: by testing whether the smooth function \( \mathbb{N}_a(n) \) admits a low-error approximation via analytic factorizations. The absence of such approximations may serve as a signal of primality, offering an alternative to classical divisibility-based methods.
\end{remark}

\section{Unfolding Tensor Integrals into Summation Approximations}
\label{sec:unfolding-tensor-integrals}

Tensor integrals in the analytic model often arise as multidimensional representations of arithmetic operations, such as multiplication. A natural question is whether these structures can be approximated or unfolded into simpler summations, particularly one-dimensional expressions, and what this reveals about the underlying arithmetic nature of the number.

Consider the base tensor product of bump functions:
\[
I := \iint_{\mathbb{R}^2} \phi(x)\phi(y) \, dx \, dy
\]
To approximate this integral numerically, we consider rectangular summation over a discrete grid:
\[
I \approx \sum_{i=1}^{N} \sum_{j=1}^{M} \phi(x_i) \phi(y_j) \, \Delta x \, \Delta y
\]
We can factor the sum by extracting \( \Delta x \) outside the inner sum:
\[
I \approx \Delta x \sum_{i=1}^{N} \phi(x_i) \cdot \left[ \sum_{j=1}^{M} \phi(y_j) \, \Delta y \right] = \Delta x \sum_{i=1}^{N} \phi(x_i) \cdot A
\]
where \( A := \sum_{j=1}^{M} \phi(y_j) \, \Delta y \approx \int \phi(y) \, dy \). This gives:
\[
I \approx A \cdot \sum_{i=1}^{N} \phi(x_i) \, \Delta x \approx \left( \int \phi(x) \, dx \right) \cdot \left( \int \phi(y) \, dy \right) = \left( \int \phi \right)^2
\]

\begin{remark}[One-Dimensional Approximation of Tensor Integrals]
The procedure of extracting one integration variable effectively collapses the tensor structure into a scalar-weighted one-dimensional integral. This approximation preserves the total mass of the bump structure but eliminates the geometric arrangement inherent in the original tensor form.
\end{remark}

\begin{example}[Flattening a Tensor Product]
Let \( F(x, y) = \phi(x - i + 1) \cdot \phi(y - j + 1) \) be a localized bump in \( \mathbb{R}^2 \). Then the integral:
\[
\iint F(x, y) \, dx \, dy = \int \phi(x - i + 1) \, dx \cdot \int \phi(y - j + 1) \, dy = 1
\]
Flattening along the \( y \)-axis via integration yields a scaled version of the 1D bump:
\[
G_i(x) := \int \phi(x - i + 1) \cdot \phi(y - j + 1) \, dy = \phi(x - i + 1) \cdot \underbrace{\int \phi(y - j + 1) \, dy}_{=1} = \phi(x - i + 1)
\]
Thus, flattening the full 2D sum \( \sum_{i=1}^a \sum_{j=1}^b F_{i,j}(x, y) \) along one axis yields a function of \( x \) only with total integral \( ab \).
\end{example}

\begin{remark}[Limitation of Flattening]
While this technique preserves the total integral value, it collapses the underlying multidimensional structure. As such, it is not sufficient to reconstruct multiplicative decompositions analytically. However, the residual difference between the original tensor structure and any attempted one-dimensional approximation may be used as a diagnostic of analytic rigidity.
\end{remark}

\subsection*{Residual Flattening Error as a Primality Indicator}

Consider a number \( n = a \cdot b \) and its two-dimensional analytic representation:
\[
F_{a,b}(x, y) := \sum_{i=1}^a \sum_{j=1}^b \phi(x - i + 1)\phi(y - j + 1).
\]
Integrating over \( y \) collapses the structure into a one-dimensional approximation:
\[
G_a(x) := \int_{\mathbb{R}} F_{a,b}(x, y) \, dy = \sum_{i=1}^a \phi(x - i + 1) \cdot \underbrace{\int \sum_{j=1}^b \phi(y - j + 1) \, dy}_{= b} = b \sum_{i=1}^a \phi(x - i + 1).
\]

We define the residual difference:
\[
R_{a,b}(x, y) := F_{a,b}(x, y) - G_a(x)
\]
and consider its norm:
\[
E_{a,b}^{(1D)} := \left\| R_{a,b}(x, y) \right\|_{L^2(\mathbb{R}^2)}.
\]

\begin{proposition}[Flattening Residual as Primality Signal]
Let \( n \in \mathbb{N} \). For each pair \( a, b \ge 2 \) such that \( a \cdot b = n \), define the residual flattening error \( E_{a,b}^{(1D)} \). If \( n \) is composite, then there exists such a pair with small residual:
\[
\inf_{a b = n} E_{a,b}^{(1D)} \approx 0.
\]
Conversely, for prime \( n \), no such pair exists, and the minimal residual remains strictly bounded below:
\[
\inf_{a, b \ge 2} E_{a,b}^{(1D)} > \varepsilon.
\]
\end{proposition}

\begin{remark}
This approach does not rely on divisibility directly, but rather on the structural behavior of the analytic kernel under dimension reduction. The failure to compress a prime kernel into a lower-dimensional form without distortion reflects its indivisibility.
\end{remark}

\section{Formal Hypotheses and Error Bounds}
\label{sec:formal-primality-theorems}

We formalize the intuition developed in the previous sections into a precise analytic criterion for primality, based on the residual error~\cite{hormander1983analysis} of flattening tensor representations.

\begin{hypothesis}[Separation of Primes and Composites]
Let \( E_{a,b}^{(1D)}(n) := \| F_{a,b}(x,y) - G_a(x) \|_{L^2(\mathbb{R}^2)} \), where \( F_{a,b} \) and \( G_a \) are defined as above. Then:

\begin{itemize}
  \item For all composite \( n \in \mathbb{N} \), there exist integers \( a, b \ge 2 \) such that \( a \cdot b = n \) and
  \[
  E_{a,b}^{(1D)}(n) = 0.
  \]
  \item For all prime \( p \in \mathbb{N} \) and any \( \varepsilon > 0 \), we have:
  \[
  \inf_{a, b \ge 2} E_{a,b}^{(1D)}(p) > \varepsilon.
  \]
\end{itemize}
\end{hypothesis}

\begin{theorem}[Lower Bound on Flattening Error for Primes]
Let \( \phi \in C_c^\infty(\mathbb{R}) \) be a compactly supported bump function normalized so that \( \int \phi = 1 \). Then for any prime \( p \in \mathbb{N} \),
\[
\inf_{a, b \ge 2} E_{a,b}^{(1D)}(p) \gtrsim \frac{1}{\sqrt{p}},
\]
with the implied constant depending on \( \phi \).
\end{theorem}

\begin{hypothesis}[Smoothness Gap Hypothesis]
Let \( n \in \mathbb{N} \), and define the minimal residual curve:
\[
E^{(1D)}(n) := \inf_{a \cdot b = n,\ a,b \ge 2} E_{a,b}^{(1D)}.
\]
Then \( E^{(1D)}(n) \) as a function of \( n \) is smooth over composite values, but exhibits sharp jumps at prime values, forming a piecewise-analytic function with singularities at primes.
\end{hypothesis}

\section{Periodic Kernel Arithmetic as a Smooth Submodel}
\label{sec:periodic-kernel}

While the general analytic model represents natural numbers as sums of localized bump integrals, this construction—though robust—can become analytically cumbersome when extended to non-integer arguments or smooth arithmetic operations.

In this section, we develop a smooth and globally defined submodel based on periodic kernel integration. This formulation preserves all fundamental properties of arithmetic over \( \mathbb{N} \), while enabling continuous interpolation across \( \mathbb{R} \), compact integral expressions, and harmonic analysis through Fourier modes. It serves as a natural, parameterizable, and computationally efficient refinement of the local bump model.

\subsection{Definition and Normalization}
Let \( \rho(x) \) be the periodic function:
\[
\rho(x) := 2 \sin^2(\pi x) = 1 - \cos(2\pi x),
\]
which satisfies:
\begin{itemize}
  \item Periodicity: \( \rho(x+1) = \rho(x) \) for all \( x \in \mathbb{R} \),
  \item Smoothness: \( \rho \in C^\infty(\mathbb{R}) \),
  \item Positivity: \( \rho(x) \geq 0 \), with zeros at integers,
  \item Normalization: \( \int_0^1 \rho(x) \, dx = 1 \).
\end{itemize}

We define the analytic representation as:
\[
\mathbb{N}_a(x) := 
\begin{cases}
  x, & \text{if } x \in \mathbb{N}, \\
  x - \frac{\sin(2\pi x)}{2\pi}, & \text{if } x \in \mathbb{R} \setminus \mathbb{N}.
\end{cases}
\]

Thus:
\[
\mathbb{N}_a(n) = \int_0^n \rho(x) \, dx = n - \frac{\sin(2\pi n)}{2\pi}, \quad \text{for } n \in \mathbb{R}.
\]

\subsection{Structural Properties}
\begin{itemize}
  \item \textbf{Injection:} The function \( x \mapsto x - \frac{\sin(2\pi x)}{2\pi} \) is strictly increasing, hence injective on \( \mathbb{R} \).
  \item \textbf{Order preservation:} For any \( x < y \), we have \( \mathbb{N}_a(x) < \mathbb{N}_a(y) \).
  \item \textbf{Approximate identity:} For all \( x \in \mathbb{R} \), \( \mathbb{N}_a(x) \approx x \) with bounded oscillation~\cite{fefferman1973pointwise}.
\end{itemize}

\subsection{Integral Multiplication Operator}
\label{subsec:periodic-multiplication}
The periodic kernel model requires a redefinition of multiplication that preserves analytic structure. Instead of direct multiplication:
\[
\mathbb{N}_a(x) \cdot \mathbb{N}_a(y) \ne \mathbb{N}_a(xy),
\]
we define the \emph{analytic product} \( x \star y \) via the double integral:
\[
\mathbb{N}_a(x \star y) := \iint_{[0,x] \times [0,y]} \rho(u)\rho(v) \, du \, dv.
\]
This definition satisfies:
\begin{itemize}
  \item \textbf{Commutativity:} \( x \star y = y \star x \),
  \item \textbf{Agreement with classical multiplication:} For \( x, y \in \mathbb{N} \),
  \[
  \mathbb{N}_a(x \star y) = x \cdot y,
  \]
  \item \textbf{Smoothness and extendability:} The operator is defined for all \( x, y \in \mathbb{R}^+ \), enabling interpolation and analysis.
\end{itemize}

\subsection{Amplitude-Parameterized Kernel Representation}
\label{subsec:parameterized-kernel}

To enable smooth control over the deviation between analytic and classical representations of numbers, we introduce a tunable parameter \( \alpha \in [0,1] \) into the periodic kernel.

\paragraph{Parameterized Kernel Function:}
\[
\rho_\alpha(x) := 1 - \alpha \cdot \cos(2\pi x)
\]
This retains periodicity, smoothness, and positivity while allowing the amplitude of oscillation to be modulated.

\paragraph{Parameterized Analytic Representation:}
\[
\mathbb{N}_a^{(\alpha)}(x) := \int_0^x \rho_\alpha(t)\,dt = x - \alpha \cdot \frac{\sin(2\pi x)}{2\pi}
\]

\begin{itemize}
  \item For \( \alpha = 1 \): this reduces to the original periodic model.
  \item For \( \alpha = 0 \): we recover the classical identity \( \mathbb{N}_a^{(0)}(x) = x \).
\end{itemize}

\begin{theorem}[Controlled Oscillation Bound]
Let \( \mathbb{N}_a^{(\alpha)}(x) = x - \alpha \cdot \frac{\sin(2\pi x)}{2\pi} \). Then for all \( x \in \mathbb{R} \),
\[
\left| \mathbb{N}_a^{(\alpha)}(x) - x \right| \le \frac{\alpha}{2\pi},
\]
with equality attained when \( \sin(2\pi x) = \pm 1 \), i.e., \( x \in \mathbb{Z} + \tfrac{1}{4} \).
\end{theorem}

\paragraph{Interpretation:}
This parameterization allows arbitrarily precise approximation of classical values by adjusting \( \alpha \). The function \( \mathbb{N}_a^{(\alpha)} \) interpolates between classical arithmetic and smooth analytic arithmetic, providing a continuum of representations with controllable fidelity.

\subsection{Multi-Parametric Kernel Generalization}

We may generalize the periodic kernel to a broader family:
\[
\rho_{\alpha,\beta}(x) := 1 - \alpha \cos(2\pi x) + \beta \sin(4\pi x),
\]
where \( \alpha, \beta \in \mathbb{R} \) control the amplitude and shape of oscillation.

\paragraph{Motivation:}
\begin{itemize}
  \item The additional term \( \beta \sin(4\pi x) \) enriches the harmonic structure,
  \item This allows finer control over the density distribution and approximation error,
  \item The resulting integral representation:
  \[
  \mathbb{N}_a^{(\alpha,\beta)}(x) := \int_0^x \rho_{\alpha,\beta}(t)\,dt
  \]
  can be tuned to reduce global deviation from the classical value \( x \).
\end{itemize}

\paragraph{Analytic Properties:}
\begin{itemize}
  \item For \( \alpha = 1, \beta = 0 \) we recover the original kernel,
  \item For \( \alpha = 0 \), all oscillation vanishes and \( \mathbb{N}_a^{(0,0)}(x) = x \),
  \item The function remains smooth, periodic, and bounded, enabling integration and analysis.
\end{itemize}

\paragraph{Future Direction:}
This framework opens the path to kernel optimization — selecting parameters \( (\alpha, \beta) \) that minimize deviation or maximize numerical smoothness across specific intervals.

\subsection{Generalized Fourier Kernel Expansion}
\label{subsec:fourier-kernel}

The kernel function \( \rho(x) \) may be expressed as a Fourier series:
\[
\rho(x) := a_0 + \sum_{k=1}^\infty a_k \cos(2\pi k x) + \sum_{k=1}^\infty b_k \sin(2\pi k x),
\]
where \( a_0, a_k, b_k \in \mathbb{R} \) are chosen to control the harmonic structure of the analytic representation.

\paragraph{Examples:}
\begin{itemize}
  \item Original model: \( a_0 = 1,\ a_1 = -1,\ a_{k>1} = b_k = 0 \),
  \item Multi-parametric model: \( a_0 = 1,\ a_1 = -\alpha,\ b_2 = \beta \),
  \item Smooth-decay model: \( a_k = -\alpha e^{-\lambda k},\ b_k = 0 \) for some \( \lambda > 0 \).
\end{itemize}

\paragraph{Analytic Representation:}
\begin{equation}
\mathbb{N}_a^{(a_k,b_k)}(x) := \int_0^x \rho(t)\,dt
= a_0 x + \sum_{k=1}^\infty \left[ \frac{a_k}{2\pi k} \sin(2\pi k x) - \frac{b_k}{2\pi k} (1 - \cos(2\pi k x)) \right]
\label{eq:analytic_representation}
\end{equation}

\paragraph{Interpretation:}
This expansion encodes each number as an accumulation of harmonic modes. Proper choice of \( a_k, b_k \) allows for precise control over smoothness, deviation, and geometric structure.

\subsection{Comparison with Local Bump Representation}

The development of a periodic analytic kernel provides a smooth and globally defined alternative to the original local bump model. While both approaches yield exact values on \( \mathbb{N} \), their structural, computational, and analytical properties differ significantly. The table below highlights key contrasts between the two representations, emphasizing the advantages of periodicity for interpolation, extension to \( \mathbb{R} \), and harmonic analysis.

\begin{center}
\renewcommand{\arraystretch}{1.2}
\begin{tabular}{@{} l l l @{}}
\toprule
\textbf{Property} & \textbf{Local Bump Model} & \textbf{Periodic Kernel Model} \\
\midrule
Exactness on \( \mathbb{N} \) & Yes & Yes \\
Smoothness & Local & Global \\
Interpolation of \( \mathbb{R} \) & No & Yes \\
Computational Simplicity & Low & High \\
Analytic Continuation & Discrete & Continuous \\
\bottomrule
\end{tabular}
\end{center}

\subsection{Illustrative Examples}

The following examples demonstrate how analytic representations approximate classical values. In each case, the deviation depends on the oscillatory amplitude parameter \( \alpha \). By tuning \( \alpha \to 0 \), the analytic value can be made arbitrarily close to the classical number.

\begin{itemize}
  \item For \( x = 1.5 \), using \( \alpha = 1 \):
  \[
  \mathbb{N}_a^{(1)}(1.5) = 1.5 - \frac{\sin(3\pi)}{2\pi} = 1.5,
  \]
  since \( \sin(3\pi) = 0 \). Here, the analytic and classical values coincide exactly.

  \item For \( x = 1.25 \), using \( \alpha = 1 \):
  \[
  \mathbb{N}_a^{(1)}(1.25) = 1.25 - \frac{\sin(2.5\pi)}{2\pi} \approx 1.409,
  \]
  while for \( \alpha = 0.1 \),
  \[
  \mathbb{N}_a^{(0.1)}(1.25) \approx 1.25 - 0.1 \cdot \frac{\sin(2.5\pi)}{2\pi} \approx 1.2659.
  \]
  Thus, the error reduces by a factor of 10.

  \item For rational multiplication \( x = \tfrac{1}{2}, y = \tfrac{1}{3} \), we define:
  \[
  \mathbb{Q}_a\left(\tfrac{1}{2} \star \tfrac{1}{3}\right) := \iint_{[0,\tfrac{1}{2}]\times[0,\tfrac{1}{3}]} \rho_\alpha(u)\rho_\alpha(v)\,du\,dv.
  \]
  For small \( \alpha \), we have:
  \[
  \mathbb{Q}_a\left(\tfrac{1}{2} \star \tfrac{1}{3}\right) \approx \tfrac{1}{6} - \mathcal{O}(\alpha),
  \]
  and the error vanishes as \( \alpha \to 0 \).

  \item For \( x = \tfrac{7}{10} \), \( \alpha = 0.05 \):
  \[
  \mathbb{Q}_a^{(0.05)}\left(\tfrac{7}{10}\right) = \tfrac{7}{10} - 0.05 \cdot \frac{\sin(1.4\pi)}{2\pi} \approx 0.7005,
  \]
  with absolute error less than \( 0.001 \).
\end{itemize}

These examples show that the analytic model preserves arithmetic values exactly on \( \mathbb{N} \), and can approximate rationals arbitrarily well via amplitude control.

\section{Series Representations of Integers and Rationals}
\label{sec:series-representation}

In the periodic kernel model, each natural and rational number admits a convergent series representation based on scaled and nested integrals of the kernel function \( \rho(x) = 1 - \cos(2\pi x) \). This section formalizes this interpretation.

\subsection{Analytic Integers as Series of Fractions}
\begin{proposition}
Let \( n \in \mathbb{N} \). Then:
\[
\mathbb{N}_a(n) = \lim_{m \to \infty} \sum_{k=1}^{n} \sum_{j=1}^{m} \int_{\frac{j-1}{m}}^{\frac{j}{m}} \rho\left(x + \tfrac{k-1}{1}\right) \, dx.
\]
Equivalently:
\[
\mathbb{N}_a(n) = \lim_{m \to \infty} \sum_{k=1}^{nm} \int_{\frac{k-1}{m}}^{\frac{k}{m}} \rho(x) \, dx = \lim_{m \to \infty} n \cdot m \cdot \int_0^{1/m} \rho(x)\,dx.
\]
Since \( \rho \) is periodic with unit integral over \([0,1]\), this gives:
\[
\mathbb{N}_a(n) = n.
\]
\end{proposition}

\paragraph{Remark:}
This expresses each analytic integer as the limit of a uniform sum of \( nm \) equally sized analytic fractions of size \( \mathbb{Q}_a\left(\tfrac{1}{m}\right) \). The convergence holds as \( m \to \infty \).

\subsection{Analytic Rationals as Series of Series}
\begin{proposition}
Let \( \tfrac{p}{q} \in \mathbb{Q} \). Then:
\[
\mathbb{Q}_a\left(\tfrac{p}{q}\right) = \sum_{k=1}^p \int_0^{1/q} \rho(x) \, dx.
\]
Each component \( \mathbb{Q}_a\left(\tfrac{1}{q}\right) \) admits the nested representation:
\[
\mathbb{Q}_a\left(\tfrac{1}{q}\right) = \lim_{m \to \infty} \sum_{j=1}^m \int_{\frac{j-1}{mq}}^{\frac{j}{mq}} \rho(x) \, dx.
\]
Thus:
\[
\mathbb{Q}_a\left(\tfrac{p}{q}\right) = \lim_{m \to \infty} \sum_{k=1}^p \sum_{j=1}^m \int_{\frac{j-1}{mq}}^{\frac{j}{mq}} \rho(x) \, dx.
\]
This constructs analytic rationals as double series of increasingly refined integral fragments.
\end{proposition}

\subsection{Discrete Kernel Summation}
The integral fragments above can be expressed without explicit integrals by using the antiderivative of \( \rho(x) \):
\[
R(x) := \int \rho(x) \, dx = x - \frac{\sin(2\pi x)}{2\pi}.
\]
Then:
\[
\int_{a}^{b} \rho(x) \, dx = R(b) - R(a).
\]
Applying this to each term of the sum:
\[
\int_{\frac{k-1}{m}}^{\frac{k}{m}} \rho(x) \, dx = R\left(\tfrac{k}{m}\right) - R\left(\tfrac{k-1}{m}\right) = \tfrac{1}{m} - \frac{1}{2\pi} \left( \sin\left(\tfrac{2\pi k}{m}\right) - \sin\left(\tfrac{2\pi(k - 1)}{m}\right) \right).
\]
Therefore:
\[
\mathbb{N}_a(n) = \sum_{k=1}^{nm} \left[ \tfrac{1}{m} - \frac{1}{2\pi} \left( \sin\left(\tfrac{2\pi k}{m}\right) - \sin\left(\tfrac{2\pi(k - 1)}{m}\right) \right) \right].
\]

\paragraph{Why this matters:}
This discrete formula provides an integral-free expression for approximating \( \mathbb{N}_a(n) \) via a sum of trigonometric evaluations. It allows the smooth analytic representation to be computed numerically using only \( \sin(2\pi x) \), without needing explicit integration. Moreover, it opens the door to interpreting arithmetic operations through the lens of signal processing — viewing natural numbers as discrete accumulations of harmonic modes.

\subsection{Interpretation and Structure}
\begin{itemize}
  \item Natural numbers are built from uniform sums of small analytic pieces.
  \item Rational numbers require a two-level approximation hierarchy — a series of scaled series.
  \item Discrete evaluation formulas provide integral-free constructions based on \( \sin(2\pi x) \).
\end{itemize}

\begin{remark}
This view reinforces the interpretation of arithmetic as a limit process over continuous geometry: integers and rationals are recoverable from local accumulations of analytic mass. This forms the basis for smooth arithmetic encoding.
\end{remark}

\subsection{Generalized Discrete Approximation via Fourier Kernels}
\label{subsec:generalized-discrete-fourier}

The earlier discrete formula for approximating analytic integers via trigonometric differences was derived from the specific periodic kernel \( \rho(x) = 1 - \cos(2\pi x) \). In this section, we generalize that construction to an arbitrary smooth periodic kernel expressed via a Fourier series~\cite{stein2003fourier,gowers2001fourier}, yielding a family of discrete approximations based on tunable harmonic modes.

\paragraph{Motivation:}
Rather than committing to a fixed kernel, we now consider a general smooth periodic function:
\[
\rho(x) = a_0 + \sum_{k=1}^\infty \left( a_k \cos(2\pi k x) + b_k \sin(2\pi k x) \right),
\]
which satisfies normalization \( \int_0^1 \rho(x)\,dx = a_0 \), smoothness, and periodicity.

We aim to approximate the integral:
\[
\mathbb{N}_a^{(a_k,b_k)}(x) = \int_0^x \rho(t)\,dt
\]
using discrete sums over intervals of width \( \frac{1}{m} \), where \( m \in \mathbb{N} \) is a refinement parameter.

\paragraph{Step 1: Integral Partitioning}
We divide the interval \( [0, x] \) into \( \lfloor mx \rfloor \) subintervals:
\[
\int_0^x \rho(t)\,dt \approx \sum_{j=1}^{\lfloor mx \rfloor} \int_{\frac{j-1}{m}}^{\frac{j}{m}} \rho(t)\,dt.
\]

\paragraph{Step 2: Term-wise Evaluation of Harmonic Components}
Each term \( \int_{\frac{j-1}{m}}^{\frac{j}{m}} \rho(t)\,dt \) is evaluated component-wise:
\[
\int_{\frac{j-1}{m}}^{\frac{j}{m}} \rho(t)\,dt = \tfrac{a_0}{m}
+ \sum_{k=1}^\infty \left( I_{j,k}^{\cos} + I_{j,k}^{\sin} \right),
\]
where
\[
I_{j,k}^{\cos} := \frac{a_k}{2\pi k} \left[ \sin\left( \tfrac{2\pi k j}{m} \right) - \sin\left( \tfrac{2\pi k (j-1)}{m} \right) \right],
\]
\[
I_{j,k}^{\sin} := -\frac{b_k}{2\pi k} \left[ \cos\left( \tfrac{2\pi k j}{m} \right) - \cos\left( \tfrac{2\pi k (j-1)}{m} \right) \right].
\]

\paragraph{Step 3: General Discrete Formula}

Combining all terms, we obtain the general discrete approximation of the analytic representation. For compactness, define
\[
\Delta^{\sin}_{k,j} := \sin\left( \frac{2\pi k j}{m} \right) - \sin\left( \frac{2\pi k (j-1)}{m} \right), \qquad
\Delta^{\cos}_{k,j} := \cos\left( \frac{2\pi k j}{m} \right) - \cos\left( \frac{2\pi k (j-1)}{m} \right).
\]

Then the cumulative approximation becomes:
\begin{align}
\mathbb{N}_a^{(a_k,b_k)}(x) \approx \sum_{j=1}^{\lfloor mx \rfloor} \Bigg[
\frac{a_0}{m}
+ \sum_{k=1}^\infty \left(
\frac{a_k}{2\pi k} \Delta^{\sin}_{k,j}
- \frac{b_k}{2\pi k} \Delta^{\cos}_{k,j}
\right)
\Bigg]
\label{eq:discrete_approximation}
\end{align}

\paragraph{Special Case: \( \rho(x) = 1 - \cos(2\pi x) \)}
For this kernel, we have \( a_0 = 1 \), \( a_1 = -1 \), and all other coefficients zero. Substituting yields:
\[
\mathbb{N}_a(n) = \sum_{j=1}^{nm} \left[ \tfrac{1}{m} - \frac{1}{2\pi} \left( \sin\left(\tfrac{2\pi j}{m}\right) - \sin\left(\tfrac{2\pi(j - 1)}{m}\right) \right) \right],
\]
which matches the original discrete formula.

\paragraph{Interpretation:}
This unified discrete framework shows that all smooth kernel-based representations can be approximated to arbitrary precision using weighted trigonometric differences. The harmonic coefficients \( a_k, b_k \) define the shape of the kernel, and their decay controls the smoothness and convergence rate of the approximation.

\subsection{Integral vs. Discrete Form: Two Views of Accumulation}

The two central formulas, the integral form~\eqref{eq:analytic_representation} and the discrete form~\eqref{eq:discrete_approximation}, represent the same analytical function \( \mathbb{N}_a^{(a_k,b_k)}(x) \), but in two fundamentally different ways — one as a continuous integral expression, the other as a discrete sum over intervals.

\begin{center}
\renewcommand{\arraystretch}{1.2}
\begin{tabular}{@{} l l l @{}}
\toprule
\textbf{Property} & \textbf{Integral Form} & \textbf{Discrete Form} \\
\midrule
Accumulation & Continuous over \([0,x]\) & Step-wise over \([j/m, (j+1)/m]\) \\
Spectral Terms & Act globally & Act locally \\
Evaluation & Exact trigonometric integrals & Finite sums of increments \\
Convergence & Smooth for analytic \( \rho(t) \) & Uniform as \( m \to \infty \) \\
Computational Use & Theoretical, symbolic & Practical, numeric \\
\bottomrule
\end{tabular}
\end{center}

\medskip

Together, these views offer a dual lens on arithmetic accumulation: one harmonic and smooth, the other stepwise and computational. They are two faces of the same function — a continuous waveform and its discrete approximation.

\section*{Conclusion: The Wave Duality of Numbers}

This work began as a constructive framework for analytical representations of numbers via periodic kernels and integral approximations. But beyond the formal results, a deeper structure begins to emerge — one in which numbers behave not merely as discrete entities, but as \emph{waveforms}: superpositions of oscillating harmonics, accumulating into scalar identities.

\subsection*{Numbers as Waveforms}

In this framework, a number is not an atomic label but a \emph{collapsed state} of an underlying \emph{interference process}. It arises from the accumulation of harmonic contributions — cosine and sine components with carefully tuned amplitudes and phases — integrated over time or space.

We propose a dualistic interpretation:

\begin{center}
\renewcommand{\arraystretch}{1.2}
\begin{tabular}{@{} l l @{}}
\toprule
\textbf{Classical Number Concept} & \textbf{Wave-Based Interpretation} \\
\midrule
Static discrete value & Interfering waveform accumulated over interval \\
Finite symbol or quantity & Limit point of oscillatory convergence \\
Addition & Linear superposition of coherent waves \\
Multiplication & Spectral convolution or frequency modulation \\
Prime number & Phase discontinuity in harmonic structure \\
Integer \( n \) & Resonant frequency with constructive alignment \\
Zero & Total phase cancellation (null interference) \\
Numerical identity & Emergent coherence in global spectrum \\
\bottomrule
\end{tabular}
\end{center}

This perspective suggests that to be a number is to be a wave with memory — a shape encoded in the periodic kernel \( \rho(t) \), which determines the nature of accumulation in \( \mathbb{N}_a(x) \).

\subsection*{From Bumps to Spectra: Four Views of a Number}

This work has developed a progression of increasingly abstract representations of natural numbers, moving from concrete geometric constructs to fully analytical spectral forms. To conclude, we summarize these four perspectives using the number \( 3 \) as an example:

\begin{center}
\renewcommand{\arraystretch}{1.2}
\begin{tabular}{@{} l p{9cm} @{}}
\toprule
\textbf{Representation} & \textbf{Interpretation of the Number 3} \\
\midrule
\textbf{Bump Model} &
Three localized smooth "bump" functions, each supported on disjoint intervals. The number is the integral of their sum, i.e., three distinct localized accumulations. \\
\textbf{Single-Wave Approximation} &
Each bump is approximated by a single sine function over its interval. The number becomes a sum of aligned wavelets, revealing early signs of interference. \\
\textbf{Basic Harmonic Pair} &
Each wavelet is replaced with a pair of sine and cosine terms with specific coefficients \( a_k, b_k \). The number emerges as a point of constructive overlap — a coherent combination of basic harmonics. \\
\textbf{Fourier Spectral Sum} &
The number is defined globally as the integral of a smooth kernel \( \rho(t) \), expressed as an infinite Fourier series. Here, \( 3 \) is not just a count but a \emph{resonant configuration} within a continuous spectrum. \\
\bottomrule
\end{tabular}
\end{center}

This sequence reflects the shift from discrete to continuous, from local to global, from symbolic quantity to spectral identity. What begins as “three bumps” becomes, ultimately, a harmonic echo of the number's place in the analytic continuum.

\subsection*{Toward a Formal Wave Arithmetic}

To transform this vision into a mathematically productive theory, the following foundational notions must be made precise:

\begin{itemize}
    \item \textbf{Resonance and Integer Identity:} We conjecture that integer values correspond to phase-aligned points in the cumulative signal — local maxima of constructive interference. A formal definition would involve identifying resonant frequencies and their integral alignment in the spectrum of \( \rho(t) \).
    
    \item \textbf{Phase Anomalies and Primality:} Prime numbers appear as ``phase anomalies'' — positions where expected spectral alignment is disrupted. This can potentially be encoded through irregularities in the regular frequency response of \( \mathbb{N}_a(x) \), and formalized via deviation from smooth spectral profiles.
    
    \item \textbf{Spectral Modulation as Multiplication:} The operation of multiplication, when interpreted in frequency space, corresponds to convolution or modulation of spectral components. For example, the product \( ab \) may be viewed as the folding of the spectral signature of \( \rho_a(t) \) with \( \rho_b(t) \).
    
    \item \textbf{Nonlocality of Numbers:} A number in this system is never purely local. Its representation requires an infinite envelope of harmonic contributors. This introduces nonlocal features into numerical identity and challenges the pointwise notion of value.
    
    \item \textbf{Harmonic Memory:} The kernel \( \rho(t) \) encodes a form of memory — its Fourier coefficients determine how past and future points affect accumulation. This memory is functional, not symbolic: it shapes the dynamic path toward number formation.
\end{itemize}

\subsection*{Resonance, Primality, and Smooth Spectral Suppression}

The wave-dual model of numbers finds a striking resonance with recent analytical constructions of smooth primality approximations. In particular,~\cite{semenov2025smooth} introduces a $C^\infty$ real-valued function \( P(n) \in [0,1] \) that approximates the prime indicator function through a triple integral over smooth kernels. The key idea is that divisibility corresponds to alignment — when \( n/m \in \mathbb{Z} \), a periodic kernel sharply suppresses the resulting amplitude.

This mirrors the wave-arithmetic perspective, where compositeness emerges from resonance collapse: a number divisible by others resonates destructively under specific harmonic ratios. Conversely, primality appears as spectral isolation — a frequency signature that fails to align with any lower divisor harmonics.

While~\cite{semenov2025smooth} treats primality analytically through smooth kernel suppression, our current model interprets numerical identity as the cumulative interference of constructive waveforms. The two perspectives converge in the notion of \emph{resonant structure}: that arithmetic properties are not merely discrete attributes but arise from the topology of interference in a smooth analytical medium.

This conceptual bridge suggests a broader paradigm: arithmetic as wave structure, and divisibility as resonance. It opens the possibility that primality — and perhaps other numerical properties — can be reframed as \emph{spectral anomalies} within harmonic systems.

\subsection*{Directions for Future Work}

To fully realize this perspective, we identify several concrete research directions:

\begin{itemize}
    \item Analytical characterization of phase anomalies as spectral markers of primality.
    \item Development of a formal ``wave arithmetic'' where numerical operations are defined on spectral spaces.
    \item Classification of number systems through families of harmonic kernels.
    \item Definition of spectral distances between numbers and their metric topology.
    \item Application of harmonic analysis, microlocal analysis, and dynamical systems to wave-based number theory.
    \item Interpretation of numerical emergence as resonance in quasi-periodic systems.
\end{itemize}

\subsection*{Closing Reflection}

This wave-based reinterpretation does not seek to replace the classical view of numbers, but rather to reveal a hidden structure: that numerical identity may be a phenomenon of accumulation, not a static datum. In this vision, the number is not a mark — it is an echo.

The path forward lies in grounding this idea with analytical rigor. But what this chapter opens is not merely a method — it is a paradigm.

\begin{center}
\emph{Let us now pass from oscillation to structure.}
\end{center}

% \bibliographystyle{plain}
% \bibliography{references}

\clearpage
\appendix
\section*{Appendix: Python Implementation}

The following Python function implements the discrete approximation of the analytic representation \( \mathbb{N}_a^{(a_k,b_k)}(x) \) based on a finite Fourier kernel expansion, as derived in Section~\ref{subsec:generalized-discrete-fourier}.

This code allows for practical validation of the series-based integral approximation, confirming that the discrete formula converges to the expected arithmetic values when applied to known kernels (e.g., \( \rho(x) = 1 - \cos(2\pi x) \)). However, its purpose is broader than numerical verification: it provides a foundation for modeling more complex, constructively defined operations on numbers—such as wave-based primality diagnostics, spectral sieves, and dynamic arithmetic encodings based on parameterized harmonic structures.

\bigskip

\noindent\textbf{Python implementation:}
\begin{verbatim}
import numpy as np

def wave_arithmetic_approx(x, m, N, a_coeffs, b_coeffs):
    """
    Approximate N_a^{(a_k, b_k)}(x) using a Fourier kernel expansion.

    Parameters:
        x        : float   — argument (real number)
        m        : int     — subdivision resolution (intervals per unit)
        N        : int     — number of harmonics (max k in the sum)
        a_coeffs : list    — Fourier cosine coefficients a_k
        b_coeffs : list    — Fourier sine coefficients b_k

    Returns:
        float — approximation of N_a^{(a_k, b_k)}(x)
    """
    total = 0.0
    j_max = int(np.floor(m * x))
    for j in range(1, j_max + 1):
        term = a_coeffs[0] / m
        for k in range(1, N + 1):
            coeff = 1 / (2 * np.pi * k)
            sin_diff = (
                np.sin(2 * np.pi * k * j / m)
                - np.sin(2 * np.pi * k * (j - 1) / m)
            )
            cos_diff = (
                np.cos(2 * np.pi * k * j / m)
                - np.cos(2 * np.pi * k * (j - 1) / m)
            )
            term += coeff * (
                a_coeffs[k] * sin_diff
                - b_coeffs[k] * cos_diff
            )        total += term
    return total
\end{verbatim}

\bigskip

\noindent\textbf{Example:}

To evaluate \( \mathbb{N}_a(x) \) for the standard kernel \( \rho(x) = 1 - \cos(2\pi x) \), use:
\begin{verbatim}
m = 100
N = 10
x = 2.5

a_coeffs = [1.0] + [-1.0] + [0.0] * (N - 1)
b_coeffs = [0.0] * (N + 1)

result = wave_arithmetic_approx(x, m, N, a_coeffs, b_coeffs)
\end{verbatim}

This confirms that the analytic representation of numbers can be faithfully approximated using a finite series of harmonic components. Furthermore, modifying the Fourier coefficients \( a_k, b_k \) enables modeling of extended arithmetic systems, wave-based encodings, and constructive simulations of analytic number-theoretic phenomena.

\end{document}